\RequirePackage{fix-cm}

\documentclass[smallextended]{svjour3}

\smartqed

\usepackage{amsfonts}
\usepackage{amsmath}
\usepackage{amssymb}
\usepackage{color}
\usepackage{graphicx}
\usepackage{mathtools}

\newtheorem{thm}{Theorem}[section]
\newtheorem{prop}[thm]{Proposition}
\newtheorem{lem}[thm]{Lemma}

\journalname{Preprint}

\begin{document}

\title{Strong Convexity in Stochastic Programs with Deviation Risk Measures}

\author{Matthias Claus \and R\"udiger Schultz \and Kai Sp\"urkel}
\authorrunning{Claus, Schultz, Sp\"urkel}

\institute{
Matthias Claus \at University of Duisburg-Essen, Thea-Leymann-Str. 9, D-45127 Essen, Germany \\
Tel.: +49 201 1836887\\
\email{matthias.claus@uni-due.de} \\
ORCID iD: 0000-0002-2617-0855
\and
R\"udiger Schultz \at University of Duisburg-Essen, Thea-Leymann-Str. 9, D-45127 Essen, Germany \\
Tel.: +49 201 1836880\\
\email{ruediger.schultz@uni-due.de}
\and
Kai Sp\"urkel \at University of Duisburg-Essen, Thea-Leymann-Str. 9, D-45127 Essen, Germany \\
Tel.: +49 201 1836883\\
\email{kai.spuerkel@uni-due.de}
}

\date{Received: date / Accepted: date}

\maketitle
\begin{abstract}
We give sufficient conditions for the expected excess and the upper semideviation of recourse functions to be strongly 
convex. This is done in the setting of two-stage stochastic programs with complete linear recourse and random right-hand
side. This work extends results on strong convexity of risk-neutral models.\\ 

\end{abstract}



\section{Introduction} \label{SecIntroduction}
Stochastic programs with linear recourse belong to those optimization problems which  are understood best, both  structurally and algorithmically. To some extent, this statement gains its substance from fairly wide spread convexity properties. 

This has spurred research into conditions for feasibility and optimality that are verifiable with reasonable effort. Another important topic in this respect is  stability of the model under data perturbation, (stochastic) approximation, or statistical  estimation. In fact, for these and any type of change  the  model could be  exposed to, it is desirable that the resulting changes in  optimal values and sets of feasible or optimal solution points stay small for small model changes.
  
Risk neutral stochastic programs with linear recourse obey ``some basic'' convexity and may fulfill a requirement called strong convexity. In the present paper, the accent is on deriving  verifiable sufficient conditions for objective functions to be strongly convex. Extending the purely expectation based analysis in \cite{Schultz1994}, we address risk aversion via deviation risk measures.

The paper  is  organized as follows. In section \ref{SecPreliminaries} preliminaries on stochastic programs and strong convexity are collected. Section \ref{SecSufficientConditions} contains the main results on strong convexity of stochastic programs with deviation risk measures. Section \ref{SecConclusion} concludes with remarks on consequences for quantitative stability of the considered two-stage models. 

\section{Preliminaries} \label{SecPreliminaries}

\subsection{Stochastic Programs with Linear Recourse} \label{SubsecStochasticPrograms}

The stochastic program with recourse is among the most popular models in optimization under uncertainty, see \cite{BirgeLouveaux2011}, \cite{ShapiroDentchevaRuszynski2014} for introduction into stochastic programming. In the two-stage linear recourse  case there is a  random linear program
\begin{equation}
\label{rando}
\min\left\{h^\top x + q^\top y \; | \; Tx + Wy = z(\omega), \; y \geq 0 \right\}
\end{equation}
with all vectors  and  matrices of  conformable (finite) dimensions. $z(\omega)$ is a random vector on some probability space $\left(\mathbb{R}^s, \mathcal{B}^s, \mu\right)$ with $\mathcal{B}^s$ the Borel $\sigma$-algebra on $\mathbb{R}^s$, and $\mu$ denotes a Borel probability measure on $\mathbb{R}^s$. The optimization problem \eqref{rando} comes with a typical requirement in stochastic programming, nonanticipativity, which  amounts to using only those information that is available upon making  decisions.

This  specifies the dynamics underlying \eqref{rando} in the following way: After having chosen  a first-stage decision $x$, the random vector $z$ is unveiled such that a well-defined optimization problem in $y$ remains.  From its set of  optimal  solutions,  a second-stage recourse solution is taken, altogether  generating the  random optimal  value 
\begin{equation}
\label{rv}
\phi(x,\omega) \coloneqq \min_y \left\{q^\top y \; | \; Wy = z(\omega) -Tx, \; y \geq  0\right\}.
\end{equation}
Clearly, when referring to decision making  under uncertainty in the  present context, it is  all about finding  a ``best  possible'' decision $x$, the only  one  to be  taken  under uncertainty (or incomplete information). Data for  making  the decision are ``compressed'' in the  real numbers $\phi(x,\omega)$  where $x$ varies  in some feasible  subset $X$ of $\mathbb{R}^s$ and $z(\omega)$ in the  support of the probability measure $\mu$.

It will be  convenient to  adopt the  perspective  about $\phi(x,\omega)$ of representing a random variable induced  by $x$. Finding  a  best  $x$ then means  finding  a  best member in the  family of random variables (= measurable  functions) $\{\phi(x,\cdot) \; | \; x\in X \}$. The  whole  history of two-stage models in  stochastic programming can be  embedded into  this view. The  most   straightforward method  of ranking random  variables  is  assigning them real  numbers  and  resorting  to the  usual   order relation on  $\mathbb{R}$. The  most popular assignment in  this  respect is  taking  the  expected  value which gives   the  optimization problem
\begin{equation}\nonumber
\min\left\{ Q_\mathbb{E}(x) \coloneqq \int_{\mathbb{R}^s} \phi(x,z)\,\mu(\text{d}z)\;|\; x\in X\right\}
\end{equation}
Historically, models  of the above  type  have marked  the  advent  of  stochastic programming, \cite{Beale1955}, \cite{Dantzig1955}, not the least due to convexity properties they obey.
Indeed, consider the value function
\begin{equation}
\label{valfu}
\varphi(t) \coloneqq \min_y \left\{q^\top y \; | \; Wy = t, \; y \geq 0 \right\}.
\end{equation}
It belongs to the folklore in stochastic programming that, under mild assumptions, namely if the feasible sets in \eqref{valfu} for all $t\in \mathbb{R}^s$ and the feasible set of the linear programming dual to \eqref{valfu}  all are nonempty, one obtains by duality that $\varphi(t) = \max_{\{\xi\,:\, W^\top\xi\le q\}}t^\top\xi$, 
confirming convexity as a pointwise maximum of linear functions. Finally, linearity of the  integral yields convexity of $Q_\mathbb{E}$.

\subsection{Strong Convexity} \label{SubsecStrongConvexity}

Convex functions have an important  role in the analysis of extremal problems, more specifically and  without being exhaustive, the investigation of structures, 
the design of algorithms, and the implementation of applications. There are various subclasses of convex functions allowing for specific conclusions, see for instance the  treatment in Chapter~7 of \cite{BauschkeCombettes2011} and in \cite{BorweinVanderwerff2010}.

The present paper is devoted to one of these subclasses, namely strongly convex functions in recourse stochastic programs. A function $f:\mathbb{R}^s \subset V \to \mathbb{R}$ defined on some convex set $V$ is called strongly convex (with modulus $\kappa$ on $V$) if there exists a real number $\kappa>0$ such that for all $x, x' \in V$ and all $\lambda$ with $0\leq\lambda\leq 1$
\begin{equation*}
f(\lambda x' + (1-\lambda) x)\;\;\le\;\;\lambda f(x')\;+\;(1-\lambda) f(x)\;-\;\frac{\kappa}{2} \lambda(1-\lambda)\|x'-x\|^2.
\end{equation*}
The literature of strong convexity appears somewhat scattered. The already mentioned books \cite{BauschkeCombettes2011} and \cite{BorweinVanderwerff2010}, the early reference \cite{Poljak1966} and the section on background material in \cite{ShapiroDentchevaRuszynski2014} all may serve as supplementary resources. In what follows we collect some material on strong convexity to indicate its importance and this way motivate our research on  verification of strong convexity for a substantial class of stochastic programs. 

Let $f: \mathbb{R}^d\to\mathbb{R}$ be a convex function. Then $f$ is strongly convex with modulus $\kappa$ if and only 
if $F(x) := f(x) - \frac{\kappa}{2}\|x\|^2$ is  convex, see \cite[Page 393]{ShapiroDentchevaRuszynski2014} for a proof. Assume for  notational convenience that $f$ is also  differentiable, and write down the following characterization of convexity for the function $F$ at some $x\in\mathbb{R}^d$
\begin{equation*}
 F(x') - F(x)\,\ge \, F'(x)(x'-x) \qquad \mbox{ for all } x, x'\in\mathbb{R}^d.
\end{equation*}
Substituting the  definition of $F$ yields
\begin{equation*}
 f(x') - f(x)- \frac{\kappa}{2}\|x'\|^2 + \frac{\kappa}{2}\|x\|^2     \,\ge \, (f'(x)-\kappa x^\intercal)(x'-x) \qquad 
 \mbox{ for all } x, x'\in\mathbb{R}^d.
\end{equation*}
After simple formula manipulation one obtains
\begin{equation}
\label{quadra}
f(x') \;\ge\; f(x) + f'(x)(x'-x) + \frac{\kappa}{2}\|x'-x\|^2\qquad 
 \mbox{ for all } x, x'\in\mathbb{R}^d. 
\end{equation}
This justifies the following interpretation: While, at any point of the interior of its domain, a  convex function is globally supported by an  affine function, a strongly convex function is so by a quadratic function with postively definite quadratic form.

Fixing $x, x'\in \mathbb{R}^d$, letting $x'$ take the role of $x$ in \eqref{quadra}, and  adding up \eqref{quadra} implies the  following characterization of strong convexity in terms  of strong  monotonicity of the  gradient
\begin{equation*}
(f'(x')- f'(x))(x'-x)   \;\geq\; \kappa \|x'-x\|^2  \qquad \mbox{ for all } x, x'\in\mathbb{R}^d.
\end{equation*}
If point $x$ in \eqref{quadra} is a minimizer of $f$ then the gradient in \eqref{quadra} vanishes, leaving
\begin{equation}
\label{growth}
f(x') \;\ge\; f(x) + \frac{\kappa}{2}\|x'-x\|^2\qquad \mbox{ for all } x, x'\in\mathbb{R}^d 
\end{equation}
what often is called the quadratic growth condition.

Beside the fact that the quadratic growth condition implies uniqueness of the minimizer it is the possibility of estimating distances of arguments by differences of objective function values that is technically attractive. In particular this  holds for estimating speed of convergence for  minimizing sequences, for deriving quantitative stability of optimization problems, and for detecting second-order asymptotics in the Sample Average Approximation Method.

\section{Sufficient Conditions for Strong Convexity of Recourse Functions} \label{SecSufficientConditions}

Before taking a look at the risk-neutral setting let us observe a  general feature:  namely using \eqref{valfu} one obtains 
\begin{equation*}
 Q_\mathbb{E}(x) = \int_{\mathbb{R}^s} \varphi(z-Tx)\,\mu(\text{d}z).
\end{equation*}
Thus, for $Q_\mathbb{E}$ to be  strongly convex it is  necessary that the null space of the  matrix $T$ is $\{0\}$. Since this assumption is unacceptably strong, the  subsequent studies will be  carried out with respect to the  transformed vector $Tx \in\mathbb{R}^s$ which we shall refer to as $x$ henceforth for the sake of simplicity. 
This turns the current object of study into
\begin{equation*}
 Q_\mathbb{E}(x) = \int_{\mathbb{R}^s} \varphi (z-x)\,\mu(\text{d}z).
\end{equation*}
with some Borel-measure $\mu$ and  $\varphi(t) = \min \{ q^\intercal y \; | \; Wy = t, \ y \geq 0 \}$. In \cite{Schultz1994} the following result regarding strong convexity of $Q_\mathbb{E}$ is shown:

\begin{thm}\label{theorem1}
Assume that the following conditions are satisfied:
\begin{itemize}
\item[A1] For every $t$ there exist some $y \geq 0$ such that $Wy = t$. (Complete recourse)
\item[A2] There exists some $\xi$ with $W^\intercal \xi < q$. (Strengthened sufficiently expensive recourse)
\item[A3] $\| z \|$ is $\mu$-integrable. (Finite first moments)
\item[A4] $\mu$ has a density $\theta$ with respect to the Lebesgue-measure and there exists a convex open set $V$, 
constants $r,\rho > 0$  such that $\theta \geq r$ a.s. on $V + B_\rho(0)$.
\end{itemize}
Then $Q_{\mathbb{E}}$ is strongly convex on $V$.
\end{thm}

The main idea of the proof is to show monotonicity of the gradient of $Q_{\mathbb{E}}$ which involves calculating the a.e. existing derivative under the integral sign and rewriting $Q'_{\mathbb{E}}$ as a Stieltjes-integral. Invoking the integration-by-parts formula for Stieltjes-integrals (which here boils down to rearranging a finite number of terms) yields an expression which can be estimated conveniently. The key arguments exploit the special structure of the linearity complex of $\varphi$ to which the following proposition from \cite{WalkupWets1969} relates to. We shall use these results throughout the text without further reference.

\begin{lem}\label{param_lemma}
Assume $A1$ and that there exists some $\xi$ such that $W^\intercal \xi \leq q$. 
It is well known from Linear Programming that under these two conditions the linear programs
\begin{align*}
& \min \{ q^\intercal y \ | \ Wy = t, y\geq 0 \} \text{ and}\\
& \max \{ t^\intercal z \ | \ W^\intercal t \geq q\}
\end{align*}
are solvable for all $t$, their optimal values coincide and the polyhedron $M_D \coloneqq \{ z \ | \ W^\intercal z \geq q\}$ is the convex hull
of its finitely many vertices $\{ d_i \ | \ i \in I \coloneqq \{ 1,\ldots,N \} \}$. We also have
\begin{itemize}
\item[(i)] $\varphi(t) = \max_{i \in I} d_i^\intercal t$.
\item[(ii)] $\varphi(t) = d_i^\intercal t$ for all $t \in K_i \coloneqq \{ z \ | \ (d_i - d_j)^\intercal z \geq 0 \text{ for all $j \in I$} \}$.
\item[(iii)] $\bigcup_{i \in I} K_i = \mathbb{R}^s$ and each intersection $K_i \cap K_j$ with $i \neq j$ is a common closed
face of dimension strictly less than $s$.
\item[(iv)] $\text{dim}(K_i \cap K_j) = s-1$ if and only if $d_i$ and $d_j$ are adjacent vertices of $M_D$.
\item[(v)] Every $K_i$ is a finite union of simplicial cones which can be written as images of $\mathbb{R}^s_+$ under
linear transformations induced by basis submatrices of $W$.
\end{itemize}
\end{lem}

We will extend the results concerning strong convexity of risk-neutral objective functions in \cite{Schultz1994} to risk-averse stochastic programs. More specifically, we will consider the expected excess over target $\eta$
\begin{align}\label{EEeq}
Q_{EE}(x) \coloneqq \int \max \{ \eta, \varphi(z-x) \}\,
\mu(\text{d}z)
\end{align}
and the upper semideviation
\begin{align*}
Q_{D_+}(x) \coloneqq \int \max\{ Q_{\mathbb{E}}(x), \varphi(z-x) \} \, \mu(\text{d}z).
\end{align*}
Due to the similar structure of $Q_{EE}$ and $Q_{D_+}$ it will be convenient to consider the function
\begin{align*}
Q_{g}(x) \coloneqq \int \max\{ g(x), \varphi(z-x) \} \, \mu(\text{d}z)
\end{align*}
where $g$ is a continuously differentiable, convex function, and derive a convenient formula for $[Q'_{g}(x+u)- Q'_{g}(x)]u$ first which is then analyzed further for the special cases $g \equiv \eta$ and $g = Q_{\mathbb{E}}$. Note that under assumptions $A1-A3$ the functions $Q_{g}(x)$ are well defined and convex on 
all of $\mathbb{R}^s$.

\subsection{Strong convexity of $Q_{EE}$} \label{SubsecStrongConvexityQEE}

In order to give a formula for $[Q'_{g}(x+u)- Q'_{g}(x)]u$ in a concise way, 
we will use the following notations:
\begin{align*}
& y_0(x,u) \coloneqq g'(x)u, \ y_i(u) \coloneqq -d_i^\intercal u, i \in I \\
& I(x,u)(\tau) \coloneqq \{ i \in I \cup \{0\} \ | \ y_i(x,u) \leq \tau \} \\
& M_{g > \varphi}(x) \coloneqq \{ z \ | \ g(x) > \varphi(z-x) \}, \ M_{g < \varphi}(x) 
\coloneqq \{ z \ | \ g(x) < \varphi(z-x) \} \\
& M_0(x) \coloneqq M_{g > \varphi}(x), \ M_i(x) \coloneqq M_{g <\varphi}(x) \cap \big( x + K_i \big), i \neq 0.
\end{align*}
We shall work with an additional assumption to ensure differentiability of $Q_g$:
\begin{itemize}
\item[A5] The sets $\{ z \ | \ 0 = \varphi(z-x) \}$ have $\mu$-measure zero for every $x$.
\end{itemize}
Geometrically speaking, A5 holds if and only if $0$ is not a vertex of $M_D$.

\begin{lem}\label{lemma1}
Assume that A1-A5 hold. Then for all $x, x+u \in V$ the following equation holds: 
\begin{align}\label{representation1}
[Q'_{g}(x+u)-Q'_{g}(x)]u =
\int_{\mathbb{R}} \mu\big(  \bigcup_{i \in I(x,u)(\tau)} M_i(x) \big\backslash 
\bigcup_{i \in I(x+u,u)(\tau)}M_i(x+u)  \big) \ \text{d}\tau.
\end{align}
\end{lem}

\begin{proof}
First note that the sets
\begin{align*}
\{ z \ | \ g(x) = \varphi(z-x) \}
\end{align*}
with $g(x) \neq 0$ all have $\mu$-measure zero which is a direct consequence
of conditions A1 and A2. The case $g(x) = 0$ is covered by assumption A5. 
In connection with A4 we can calculate $Q'_{g}(x)u$ by differentiation under the integral sign as
\begin{align}
\label{imp-formula1}
Q'_{g}(x)u = \mu(M_{g > \varphi}(x))g'(x)u - \sum_{i \neq 0}
\mu( M_i(x) ) \, d_i^\intercal u,
\end{align}
Following the argument in [1], we see that for fixed $x$ and $u$ we can rewrite \eqref{imp-formula1} as the expected value of a discrete probability distribution with values $y_0 \coloneqq g'(x)u, \, y_i \coloneqq -d_i^\intercal u$
and probabilities $\pi_0 \coloneqq \mu(M_{g > \varphi}(x)), \, \pi_i \coloneqq \mu(M_i(x))$ respectively. By transformation to the induced measure and using integration by parts for Riemann-Stieltjes integrals we arrive at
\begin{align*}
[Q'_{\mathcal{D}_+}(x+u)-Q'_{\mathcal{D}_+}(x)]u = \int (F_x(\tau) - F_{x+u}(\tau)) \, \text{d}\tau, 
\end{align*}
where $F_x = \sum_{i \in I_{x,u}(\tau)} \pi_i$ is the cdf of the discrete distribution just defined. Note that the boundary terms in the integration by parts-formula cancel out because $(F_x-F_{x+u})(\tau) = 0$ for sufficiently big $|\tau|$. \\

For proving the lemma we will demonstrate that 
\begin{align*}
F_x(\tau) - F_{x+u}(\tau) = \mu\big(  \bigcup_{i \in I(x,u)(\tau)} M_i(x) \big\backslash 
\bigcup_{i \in I(x+u,u)(\tau)}M_i(x+u)  \big)
\end{align*}
which can be done by showing that for all $\tau$
\begin{align}
\label{inclusion1}
 \bigcup_{i \in I(x+u,u)(\tau)} M_i(x+u) \subset \bigcup_{i \in I(x,u)(\tau)} M_i(x).
\end{align}
To see why this sufficient, note that by lemma \ref{param_lemma} and A4
\begin{align*}
F_x(\tau) = \sum_{i \in I(x,u)(\tau)}\mu(M_i(x)) =   \mu\big( \bigcup_{i \in I(x,u)(\tau)} M_i(x)\big)
\end{align*}
and thus, if \eqref{inclusion1} holds,
\begin{align*} 
F_x(\tau) - F_{x+u}(\tau) 
& = \sum_{i \in I(x,u)(\tau)}\mu(M_i(x)) - \sum_{i \in I(x+u,u)(\tau)}\mu(M_i(x+u)) \\
& = \mu\big( \bigcup_{i \in I(x,u)} M_i(x) \big\backslash \bigcup_{i \in I(x+u,u)} M_i(x+u) \big).
\end{align*}
From the proof in \cite[theorem 2.2]{Schultz1994} it is known that 
\begin{align}\label{inclusion2}
\bigcup_{i \in I(x+u,u)(\tau) \backslash \{0\}} u + K_i \subset \bigcup_{i \in I(x,u)(\tau) \backslash \{0\}} K_i
\end{align}
and we shall use this result to simplify the proof of inclusion \eqref{inclusion1}. \\

Let $z_0 \in \bigcup_{i \in I(x+u,u)(\tau)} M_i(x+u) \Big \backslash 
\bigcup_{i \in I(x,u)(\tau)} M_i(x)$, then we have two cases to consider: \\

case 1) $z_0 \in M_0(x+u)$, from which follows $g'(x+u)u \leq \tau$. By convexity of $g$ we conclude
$g'(x)u \leq \tau$, so $0 \in I(x,u)(\tau)$ and we have  $z_0 \notin M_0(x)$ by assumption. We thus have $g(x) < \varphi(z-x) = d_{i_1}^\intercal (z-x)$ for a suitable index $i_1 \neq 0$ which needs to be such that $-d_{i_1}^\intercal u > \tau$, otherwise $i_i \in I(x,u)(\tau)$. On the other hand we have
\begin{align*}
g(x+u) &> \varphi(z-x-u) \geq d_{i_1}^\intercal (z-x-u) = d_{i_1}^\intercal (z-x) - d_{i_1}^\intercal u  \\ 
& > g(x) - d_{i_1}^\intercal u \geq g(x+u) - g'(x+u)u - d_{i_1}^\intercal u \\
& \geq  g(x+u) - \tau -  d_{i_1}^\intercal u,
\end{align*}
so we have $-  d_{i_1}^\intercal u < \tau$, which is a contradiction. So case 1) is impossible. \\

case 2) $z_0 \in M_{i_0}(x+u)$ for some index $i_0 \neq 0$ (which means 
$g(x+u) < \varphi(z-x-u) = d_{i_0}^\intercal (z-x-u)$ and $-d_{i_0}^\intercal u \leq \tau$). By \eqref{inclusion2} it follows $z_0 \in M_0(x)$ and
by assumption we have $g'(x)u > \tau$, from which we conclude by convexity of
$g$ that also $g'(x+u)u > \tau$. From here we get with the same arguments as above
\begin{align*}
d_{i_0}^\intercal (z-x) - d_{i_0}^\intercal u &= d_{i_0}^\intercal (z-x-u) = \varphi(z-x-u) 
> g(x+u) \\
& \geq g(x) + g'(x)u > g(x) 
+ \tau  > \varphi(z-x) + \tau \\
& \geq d_{i_0}^\intercal (z-x) + \tau,
\end{align*}
which yields $-d_{i_0}^\intercal u > \tau$ in contradiction to what was stated above.
\end{proof}

For fixed $x$ and $u$, the integrand on the right-hand side of \eqref{representation1} is piecewise constant in $\tau$. The proof of strong convexity will consist of restricting the integration to a suitable subset where the integrand is constant 
in $\tau$ and minorized by a quantity that can be estimated easier. The next lemma gives the proof of formula (2.13) in \cite[theorem 2.2]{Schultz1994} concerning estimates of this integrand:
\begin{lem}
\label{lemma2}
Let $K = \{u \ | \ c_j^\intercal u \geq 0, \ j \in J \}$ be a pointed, polyhedral cone in $\mathbb{R}^s$ 
(in particular $|J| < \infty$) with  $\text{int}(K) \neq \emptyset$.
Then there is some $\alpha > 0$ such that
\begin{align*}
\inf_{ u \in K } \,
\sup_{ j \in J}  \, c_j^\intercal u \, \geq \alpha \|u\|.
\end{align*}
\end{lem}

\begin{proof}
For arbitrary $0 \neq u \in K$ it holds that $\frac{u}{\|u\|} \in K$ due to $K$ being a cone, so it is enough to prove that
\begin{align*}
\inf_{ \substack{u \in K \\ \|u\| = 1} }
\sup_{j} c_j^\intercal u = \alpha > 0.
\end{align*}
From the definition of $K$ it directly follows that $\alpha \geq 0$. If $\alpha$ was $0$, we could find a sequence $\{u_n\}_{n \geq 1}$ in the compact set 
$K_1 \coloneqq K \cap \{ u \ | \ \|u\| = 1 \}$ such that
\begin{align*}
\sup_{ j } c_j^\intercal u_n \rightarrow 0,
\end{align*}
and by passing to a subsequence we had $u_{n_k} \rightarrow u_0 \in K_1$ and by continuity $c_j^\intercal u_0 = 0$ for all indices $j$. That however would be a contradiction to the full-dimensionality of $K$ because in particular 
$\|u_0\|=1. \ \ \square$
\end{proof}

The main result concerning strong convexity of the expected excess $Q_{EE}$
is given in the next theorem:
\begin{thm}[Strong convexity for $Q_{EE}$]\label{theorem2}
Assume that $A1-A5$ hold. Then there exists some $c > 0$ such that
for all $\eta < c$ we have strong convexity of the expected excess $Q_{EE}$ over target $\eta$ on the open, 
convex set $V$ (cf. A4). 
\end{thm}

Before the proof we shall look at the geometry of the expected excess
and give a heuristic reasoning of why the theorem should hold:
\begin{center}
\includegraphics[width=0.8\textwidth,trim=0 60 0 140, clip=true]{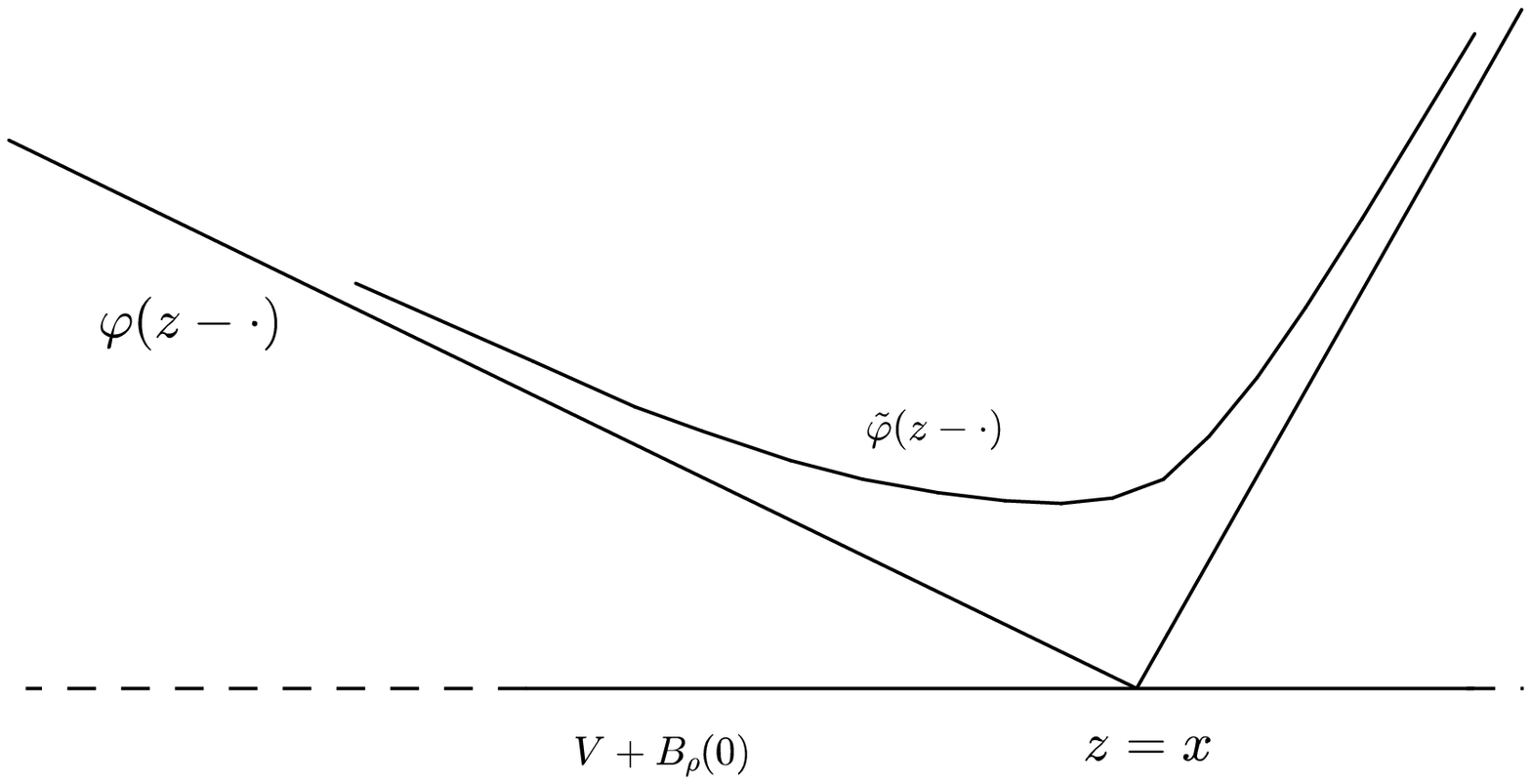}
\end{center}
The picture shows the one-dimensional case with $\varphi \geq 0$ und $\eta < 0$,
so that $Q_{EE} = Q_\mathbb{E}$.
For fixed $z \in V$ the function $\varphi(z-\cdot)$ can be approximated 
by a $\mathcal{C}^2$-function $\tilde{\varphi}(z-\cdot)$ which has
second derivative $\varphi''(z-x) \geq C > 0$ with $x$ close to $z$. 
With condition A4 in mind we get
\begin{align*}
&Q_\mathbb{E}''(x) = \int \varphi''(z-x) \, \mu(\text{d}z)
= \int \theta(t) \, \varphi''(t-x) \, \text{d}t \\
&\geq \int_{B_\rho(x)} \theta(t) \, \varphi''(t-x) \, \text{d}t 
\geq \int_{B_\rho(x)} r \, C \, \text{d}t > 0.
\end{align*}
If $\eta$ is gradually increased, the graph of the integrand 
in \eqref{EEeq},\\
$\max\{ \eta, \varphi(\cdot-x) \}$, is truncated as in the next picture: 
\begin{center}
\includegraphics[width=0.8\textwidth,trim=0 60 0 160, clip=true]{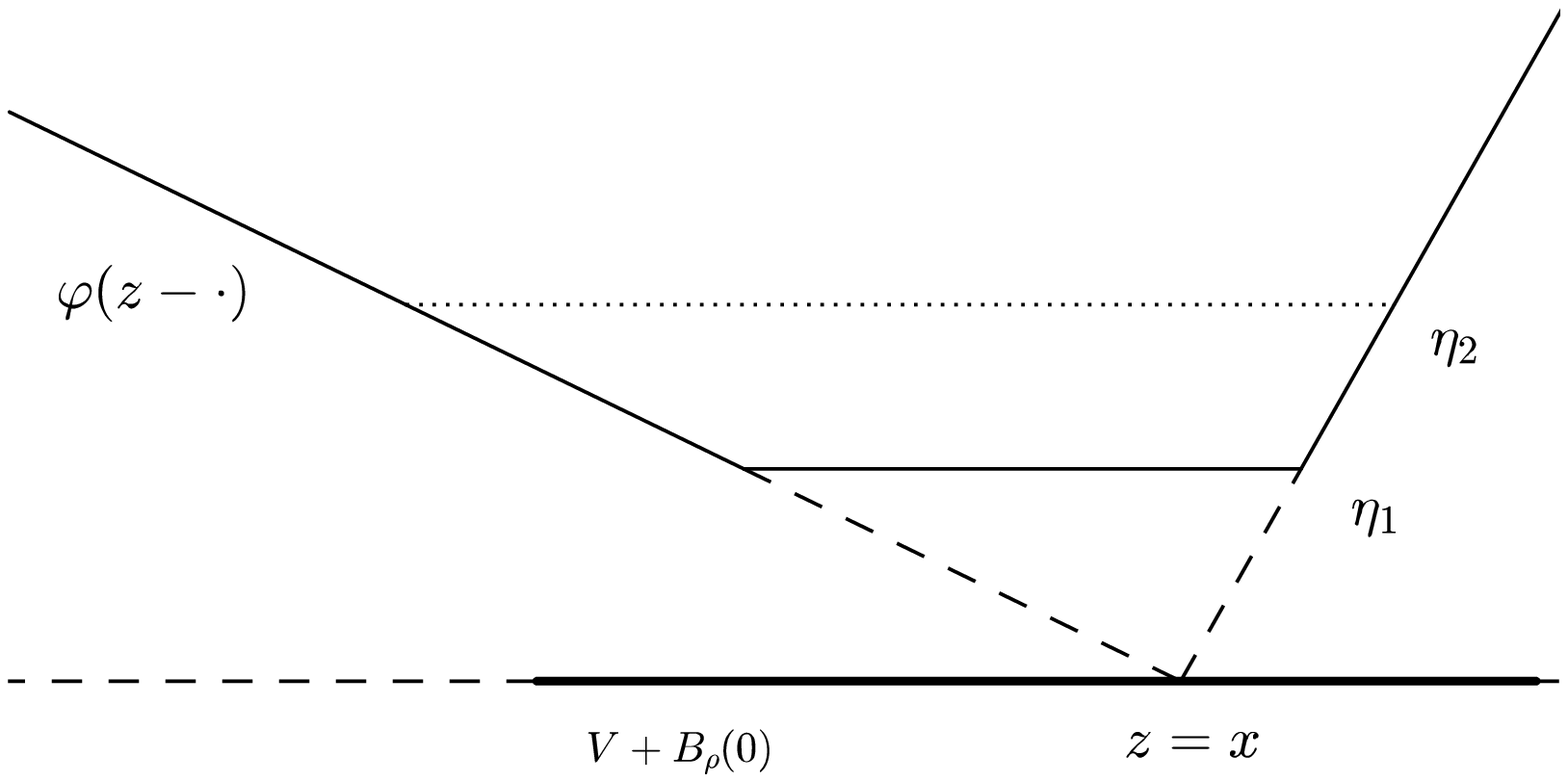}
\end{center}
As long as the "bumps" in the graph lie above the set $V + B_\rho(0)$, the 
approximation argument from before yields strong convexity of $Q_{EE}$,
for the approximants of $\max\{ \eta, \varphi(z-\cdot) \}$ will be curved
close to the bumps. 
Note that the best possible modulus of strong convexity will be smaller 
the greater $\eta$ is chosen. When $\eta$ exceeds a certain threshold
the graph of $\max\{ \eta, \varphi(z-\cdot) \}$ will be flat in all of
$V + B_\rho(0)$, so strong convexity of $Q_{EE}$ can no longer be expected
to hold.\\

\begin{proof}[Theorem]
By lemma \ref{param_lemma} (iii) we have $\mathbb{R}^s = \bigcup_{i \in I} K_i$ with full-dimensional, pointed polyhedral cones $K_i$. Define $K_i^+ = \{ z \in K_i \ | \ d_i^\intercal z \geq 0 \}$ and $K_i^- = \{ z \in K_i \ | \ d_i^\intercal z \leq 0 \}$,
so we can write
\begin{align*}
K_i = K_i^+ \cup K_i^-,
\end{align*}
and at least one of the two cones $K_i^+$ and $K_i^-$ has nonempty interior.\\
For arbitrary $u \neq 0$ we have $u \in K_i$ for some $i \in I$. Consider the case $u \in K_i^+$ when $\text{int}(K_i^+) \neq \emptyset$ first. The other case will be reduced to this one later. Note that under assumptions A1 and A2 $K_i$ is a polyehdral cone generated by the vectors $d_i - d_j$, where all $d_j$ are adjacent to $d_i$.  Lemma \ref{lemma2} can then be applied to yield some $\alpha> 0$, which - due to there being only finitely many $K_i$ - can be chosen independently of $u$ and $i$, and some index $j$ such that 
\begin{align*}
\alpha \|u\| \leq (d_{i} - d_{j})^\intercal u,  
\end{align*}
from which follows
\begin{align*}
-d_{i}^\intercal u + \alpha \|u\| \leq -d_{j}^\intercal u.
\end{align*}
By lemma \ref{lemma1} we get (note that the index sets $I(x,u)(\tau)$ introduced above lemma \ref{lemma1} do not depend on $x$ in the case of $g \equiv \eta$)
\begin{align}
&[Q'_{EE}(x+u)-Q'_{EE}(x)]u \nonumber \\
& = \int_{\mathbb{R}} \mu\big(  \bigcup_{l \in I(u)(\tau)} M_l(x) \big\backslash 
\bigcup_{l \in I(u)(\tau)}M_l(x+u)  \big) \ \text{d}\tau \nonumber \\
& \geq \int_{-d_{i}^\intercal u}^{-d_{i}^\intercal u + \alpha \|u\|} 
\mu\big(  \bigcup_{l \in I(u)(\tau)} M_l(x) \big\backslash 
\bigcup_{l \in I(u)(\tau)}M_l(x+u)  \big) \ \text{d}\tau. \label{EE1}
\end{align}
Because of $u \in K_i^+$ we have $d_i^\intercal u \geq 0$, so there two cases to consider: $-d_i^\intercal u \leq 0 \leq -d_i^\intercal u + \alpha \|u\|$ and $-d_i^\intercal u + \alpha \|u\| \leq 0$. Given $-d_i^\intercal u \leq 0 \leq -d_i^\intercal u + \alpha \|u\|$ it has to be $d_i^\intercal u \geq \frac{1}{2}\alpha \|u\|$
or $-d_i^\intercal u + \alpha \|u\| \geq \frac{1}{2}\alpha \|u\|$. Assuming the latter we can continue \eqref{EE1} as follows:
\begin{align*}
& \geq \int_{0}^{-d_{i}^\intercal u + \alpha \|u\|} 
\mu\big(  \bigcup_{l \in I(u)(\tau)} M_l(x) \big\backslash 
\bigcup_{l \in I(u)(\tau)}M_l(x+u)  \big) \ \text{d}\tau \\
& =  \int_{0}^{-d_{i}^\intercal u + \alpha \|u\|} 
\mu\big(  \bigcup_{l \in I(u)(\tau)}x + K_l \big\backslash 
\bigcup_{l \in I(u)(\tau)} x + u + K_l  \big) \ \text{d}\tau \\
& \geq \int_{0}^{-d_{i}^\intercal u + \alpha \|u\|} 
\mu\big(  x + K_i \big\backslash 
\bigcup_{-d_l^\intercal u < -d_j^\intercal u} x + u + K_l  \big) \ \text{d}\tau \\
& \geq \frac{1}{2}\alpha \|u\| \, \mu\big(  x + K_i \big\backslash 
\bigcup_{-d_l^\intercal u< -d_j^\intercal u} x + u + K_l  \big).
\end{align*}
From here on it can be argued as in \cite[page 10]{Schultz1994} that there exists some $\beta > 0$ only depending on $i$ such that
\begin{align*}
\mu\big(  x + K_i \big\backslash 
\bigcup_{-d_l^\intercal u < -d_j^\intercal u} x + u + K_l  \big) \geq \beta \|u\|
\end{align*}
which concludes the analysis of the first case.\\
Going back to the other cases we will always have $d_i^\intercal u \geq \frac{1}{2}\alpha \|u\|$ from here on.
We estimate \eqref{EE1} as
\begin{align*}
& \geq \int_{-d_{i}^\intercal u}^{\min\{ 0, -d_i^\intercal u + \alpha \|u\| \}} 
\mu\big(  \bigcup_{l \in I(u)(\tau)} M_l(x) \big\backslash 
\bigcup_{l \in I(u)(\tau)}M_l(x+u)  \big) \ \text{d}\tau \\
& \geq \int_{-d_{i}^\intercal u}^{\min\{ 0, -d_i^\intercal u + \alpha \|u\| \}} 
\mu\big( M_i(x) \big\backslash 
\bigcup_{-d_l^\intercal u < 0}M_l(x+u)  \big) \ \text{d}\tau \\
& \geq \frac{1}{2}\alpha \|u\| \ \mu\big(   M_i(x) \big\backslash \bigcup_{-d_l^\intercal u < 0}M_l(x+u)  \big) 
\end{align*}
Note that the last inequality follows because $-d_i^\intercal u < \tau < 0$ for each $\tau$ in question and thus
\begin{align*}
 \bigcup_{l \in I(u)(\tau)} M_l(x) &\supset M_i(x), \\
 \bigcup_{l \in I(u)(\tau)}M_l(x+u) &\subset \bigcup_{-d_l^\intercal u < 0}M_l(x+u),
\end{align*}
and the sets on the right-hand side do not depend on $\tau$ anymore.\\

It remains to be shown that there exists some $\beta > 0$ independent of $x$ and $u$ such that 
\begin{align}
\label{EE6}
\mu\big(   M_i(x) \big\backslash \bigcup_{-d_l^\intercal u < 0}M_l(x+u)  \big) \geq \beta \|u\|.
\end{align}
Without loss of generality assume $\eta > 0$. It can easily be shown that the set
\begin{align*}
M_i^+(x) \coloneqq \big( x+K_i^+ \big) \cap \{ z \ | \ \varphi(z-x) \geq \eta \} 
\end{align*}
is a full-dimensional polyhedron with vertices and with $s-1$-dimensional facet
\begin{align*}
F \coloneqq 
\big( x+K_i^+ \big) \cap \{ z \ | \ \varphi(z-x) = \eta \}.
\end{align*}
The situation can be pictured like this:
\begin{center}
\includegraphics[width=0.9\textwidth]{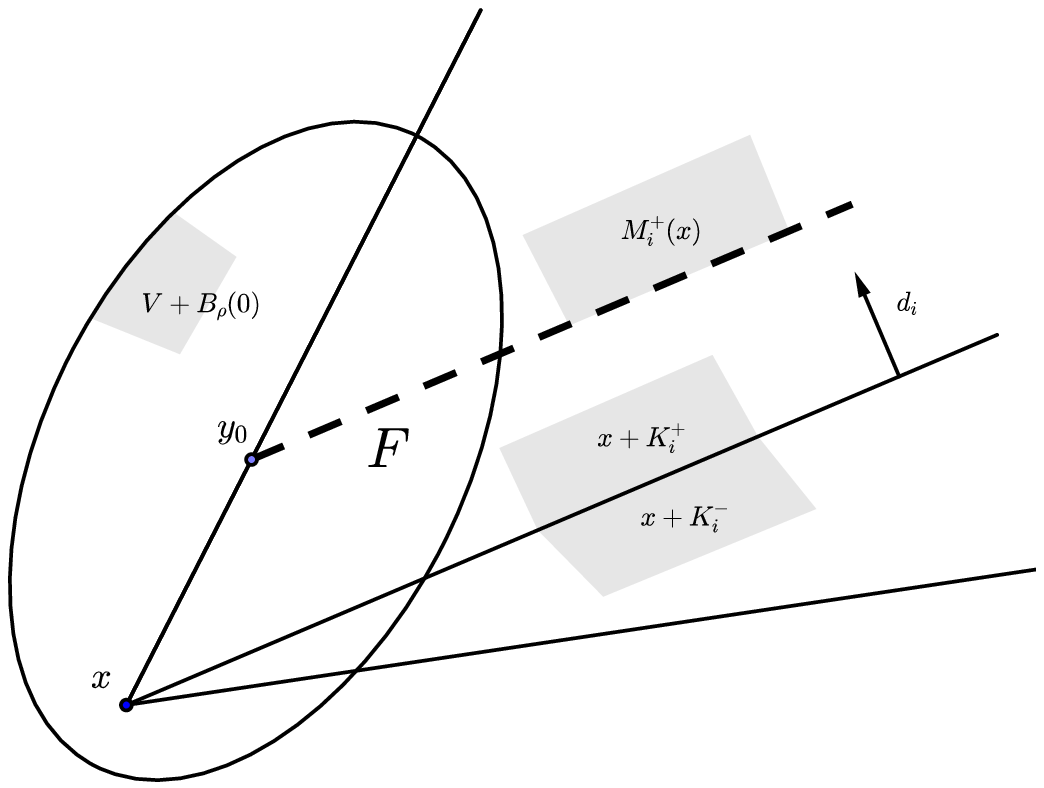}
\end{center}
With condition A4 in mind assume $\eta$ is small enough so that for at least one vertex $y_0$ of $M_i^+(x)$ contained in $F$ we have $\| y_0 - x \| = \rho_0 < \rho$. The set
\begin{align}
\label{EE7}
F_0 \coloneqq F \cap B_{\rho - \rho_0}(y_0)
\end{align}
then has positive $s-1$-dimensional measure. We will now show that
\begin{align}
\label{EE3}
\tilde{F} \coloneqq \bigcup_{0 \leq \lambda < 1} \lambda u + F_0 \subset \big( M_i^+(x) \big\backslash 
\bigcup_{-d_l^\intercal u < 0}M_l(x+u) \big).
\end{align}
The inclusion $\bigcup_{0 \leq \lambda < 1}\lambda u + F_0 \subset M_i^+(x)$ directly follows from $u \in K_i^+$, $K_i^+$ being a convex cone and the definition of $F_0$. \\

For arbitrary $l$ such that $-d_l^\intercal u < 0$, $0\leq \lambda < 1$ and $\tilde{z} \in F_0$, which we
can write as $y_0 + z'$ with $d_i^\intercal z' = 0$,  we have (note that $\tilde{z} - x \in K_i^+$)
\begin{align*}
d_l^\intercal (\lambda u + \tilde{z} - x -u ) &= (1-\lambda) (-d_l^\intercal u) + d_l^\intercal (\tilde{z} - x)\\
&\leq (1-\lambda) (-d_l^\intercal u) + d_i^\intercal (\tilde{z} - x)\\
&< d_i^\intercal (y_0 - x) +  d_i^\intercal z' = \eta 
\end{align*}
so that $\lambda u + \tilde{z} \notin M_l(x+u)$. Thus \eqref{EE3} is shown.\\

As the last step in the proof we will show that
\begin{align}
\label{EE4}\tilde{F} \subset V + B_\rho(0),
\end{align}
$V$ being the set from condition A4, and that there is some $\beta > 0$ such that
\begin{align}
\label{EE5}
\lambda ( \tilde{F} ) \geq \beta \|u\|.
\end{align}
From \eqref{EE3}, \eqref{EE4} and \eqref{EE5} follows \eqref{EE6} as follows:
\begin{align*}
\mu\big(   M_i(x) \big\backslash \bigcup_{-d_l^\intercal u < 0}M_l(x+u)  \big) 
 \geq \mu( \tilde{F}) 
= \int_{\tilde{F}} g(t) \ \text{d}t 
\geq r \ \lambda(\tilde{F})
\geq r \beta \|u\|.
\end{align*}
Inclusion \eqref{EE4} holds because $x + \lambda u \in V$ for all $0 \leq \lambda < 1$ and thus by \eqref{EE7}
\begin{align*}
&\lambda u + F_0 = (x + \lambda u) + (y_0 - x) - y_0 + F_0 \\
&\subset V + 
\text{cl}\big(B_{\rho_0}(0)\big)
+ B_{\rho - \rho_0}(0) \subset V + B_{\rho}(0).
\end{align*}
Estimate \eqref{EE5} is correct because we have $d_i^\intercal u \geq \frac{1}{2}\alpha \|u\|$, $F_0$ has positive $s-1$-dimensional volume and the density of $\mu$ is bounded below on $V + B_\rho(0)$ by a positive constant. Applying Cavalieri's principle yields the desired result. To get an idea of
the geometry behind these arguments, a picture might be helpful:
\begin{center}
\includegraphics[width=1.0\textwidth]{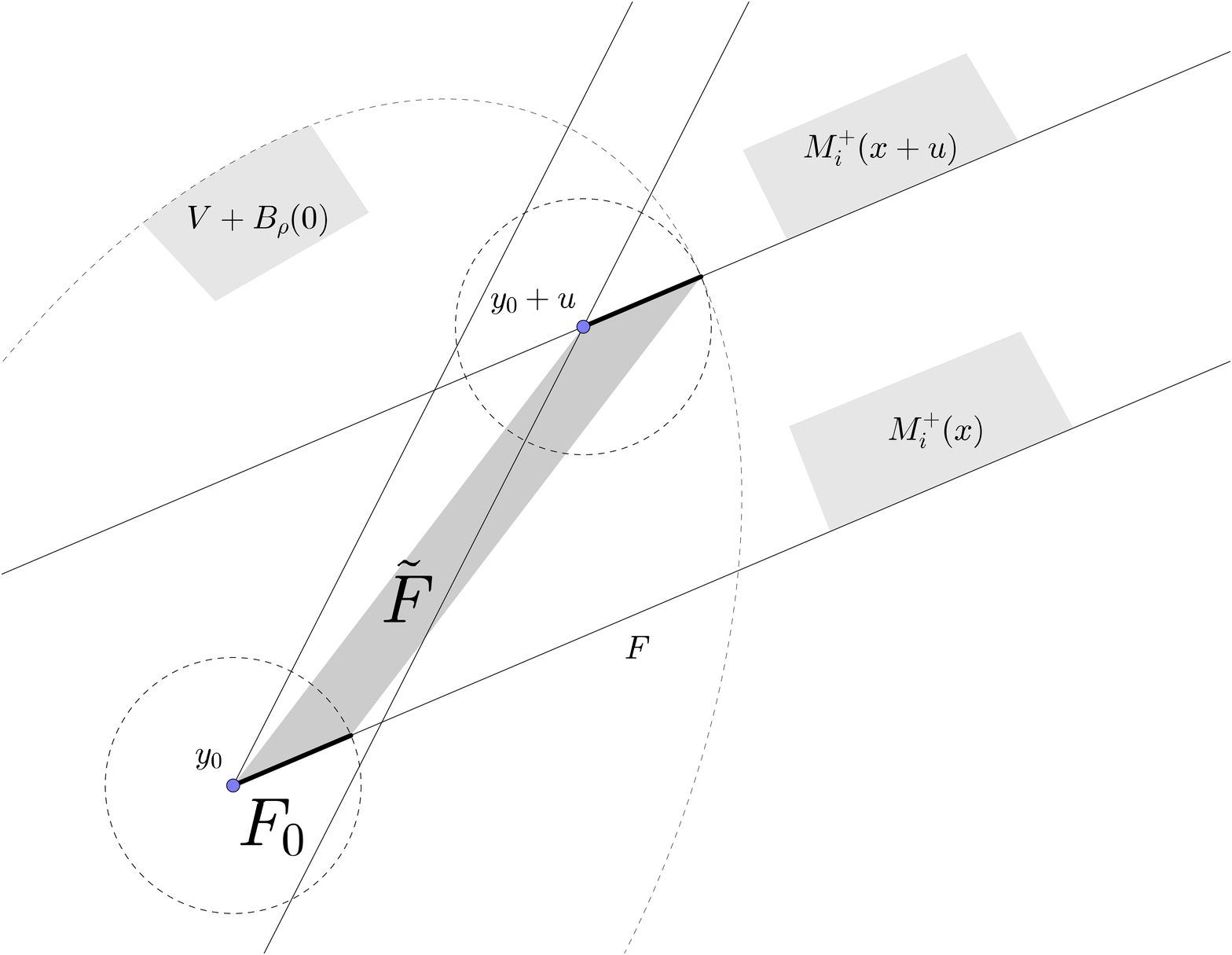}
\end{center}
Consider now the case $u \in K_i^-$ when $\text{int}(K_i^-) \neq \emptyset$:
It then holds $-u \in K_{i_0}^+$ for some $i_0$ with $\text{int}(K_{i_0}^+) \neq \varnothing$. Again invoking lemma \ref{lemma2} yields some $\alpha > 0$ 
and some index $j_0$ such that $d_{j_0}$ and $d_{i_0}$ are adjacent with
\begin{align*}
\alpha \|u\| \leq (d_{i_0} - d_{j_0})^\intercal (-u),  
\end{align*}
from which follows
\begin{align}
\label{EE8}
-d_{i_0}^\intercal u - \alpha \|u\| \geq -d_{j_0}^\intercal u.
\end{align}
By lemma \ref{lemma1} we thus get
\begin{align*}
&[Q'_{g}(x+u)-Q'_{g}(x)]u \\
& = \int_{\mathbb{R}} \mu\big(  \bigcup_{l \in I(u)(\tau)} M_l(x) \big\backslash 
\bigcup_{l \in I(u,u)(\tau)}M_l(x+u)  \big) \ \text{d}\tau \\
& \geq \int_{-d_{i_0}^\intercal u - \alpha \|u\|}^{-d_{i_0}^\intercal u} 
\mu\big(  \bigcup_{l \in I(u)(\tau)} M_l(x) \big\backslash 
\bigcup_{l \in I(u,u)(\tau)}M_l(x+u)  \big) \ \text{d}\tau. \\
\end{align*}
The integrand can be rewritten as follows (note that for $i \neq j$ the sets $M_i(y) \cap M_j(y)$ have $\mu$-measure $0$ for all choices of $y$):
\begin{align*}
&\mu\big(  \bigcup_{l \in I(u)(\tau)} M_l(x) \big\backslash 
\bigcup_{l \in I(u,u)(\tau)}M_l(x+u)  \big) \\
& = \mu\big( \ \big( \mathbb{R}^s \big\backslash \bigcup_{\{l \ | \ y_l(u)>\tau\}} M_l(x) \big) \big\backslash 
\big(  \mathbb{R}^s \big\backslash \bigcup_{\{l \ | \ y_l(u,u)>\tau\}} M_l(x+u)  \big) \ \big) \\
& = \mu\big(  \bigcup_{\{l \ | \ y_l(u,u)>\tau\}} M_l(x+u) 
\big\backslash \bigcup_{\{l \ | \ y_l(u)>\tau\}} M_l(x)  \big) \\
& \geq \mu \big(  M_{i_0}(x+u) \big\backslash \bigcup_{\{ l \ | \ y_l(u) > -d_{j_0}^\intercal u \}} M_l(x) \big).
\end{align*}
The last inequality holds because by \eqref{EE8} we have $-d_{j_0}^\intercal u \leq -d_{i_0}^\intercal u - \alpha \|u\| < \tau < -d_{i_0}^\intercal u$. By renaming $x' = x+u, \ u' = -u $, observing $\{ l \ | \ y_l(u) > -d_{j_0}^\intercal u \} = \{ l \ | \ y_l(u') < -d_{j_0}^\intercal(u') \}$ and noting that the lower estimate for the integrand does not depend on $\tau$ we can continue the estimate for the integral above as
\begin{align*}
[Q'_{g}(x' + u')-Q'_{g}(x')]u' \geq \alpha \|u\| \ \mu\big(  M_{i_0}(x') \big\backslash \bigcup_{\{ l \ | \ y_l(u',u') < -d_{j_0}^\intercal u' \}} M_l(x'+u') \big),
\end{align*}
where $x', x'+u' \in V$ with $u' \in K_{i_0}^+$ such that $\text{int}(K_{i_0}) \neq \emptyset$.
The desired estimate now follows as demonstrated previously after \eqref{EE1}.
\end{proof}

\subsection{Strong convexity of $Q_{D_+}$} \label{SubsecStrongConvexityQD}

\begin{thm}[Strong convexity for the upper semideviation $Q_{D_+}$] \label{ThSemidev}
Under assumptions A1-A5 there exists some continuous function 
$C : \mathbb{R}^s \rightarrow [0,\infty)$ with $C(u) \rightarrow 0$, as 
$\|u\| \rightarrow 0$, such that $C(u) > 0$ for all $u \neq 0$ and
\begin{align*}
[Q'_{D_+}(x+u)-Q'_{D_+}(x)]u \geq C(u) \, \|u\|^2
\end{align*}
for all $x, x+u \in V$.
If we have the additional assumption
\begin{itemize}
\item[A6] $0 \leq q$
\end{itemize} 
then $Q_{D_+}$ is a strongly convex function on $V$.
\end{thm}

\begin{proof}
Let $x, x+u \in V$. 
We apply lemma \ref{lemma1} on $g = Q_{\mathbb{E}}$ to get
\begin{align}
&[Q'_{D_+}(x+u)-Q'_{D_+}(x)]u = [Q'_{Q_{\mathbb{E}}}(x+u)-Q'_{Q_{\mathbb{E}}}(x)]u \nonumber \\
&= \int \mu\big( \bigcup_{l \in I(x,u)(\tau)} M_l(x) \big\backslash 
\bigcup_{l \in I(x+u,u)(\tau)} M_l(x+u) \big) \ \text{d}\tau \nonumber \\
&\geq \int_{  Q_{\mathbb{E}}'(x)u }^{Q'(x+u)u} \mu\big( \bigcup_{l \in I(x,u)(\tau)}
M_l(x) \big\backslash \bigcup_{l \in I(x+u,u)(\tau)}M_l(x+u) \big) \ \text{d}\tau. \label{SD1}
\end{align}
By A1 and A2 we can apply lemma \ref{param_lemma} as we did in the last proof: The fan of cones $\{K_l\}$ covers $\mathbb{R}^s$, so $u \in K_i$ and $-u \in K_j$ for suitable indices $i,j \in I$. By lemma \ref{lemma2} we can find some constant $\alpha>0$ independent of $x$ and $u$ such that $(d_i-d_j)^\intercal u \geq \alpha \|u\|$;
condition A4 and the full-dimensionality of all $K_l$ imply that there exists some $\delta > 0$ such that for all $y \in V$ and all indices $l \in I$ we have $\mu(y+K_l) \geq \delta$. We thus have
\begin{align}
& -d_j^\intercal u - Q_{\mathbb{E}}'(x+u)u = 
\sum_{l=1}^N \mu(x+u+K_l) (d_l - d_j)^\intercal u \nonumber \\ 
&\geq \mu(x+u+K_i) \, (d_i-d_j)^\intercal u \geq 
\delta \alpha \|u\|, \nonumber \\
& Q_{\mathbb{E}}'(x)u - (-d_i^\intercal u ) =
\sum_{l=1}^N \mu(x+K_l) (d_i - d_l)^\intercal u \nonumber \\
&\geq \mu(x+K_j) \, (d_i-d_j)^\intercal u \geq 
\delta \alpha \|u\|. \label{SD2}
\end{align} 
In \eqref{SD1} it is $Q_{\mathbb{E}}'(x)u < \tau < Q_{\mathbb{E}}'(x+u)u$, so by \eqref{SD2} we get $0,i  \in I(x,u)(\tau)$ implying for all $\tau$ in question the inclusion
\begin{align*}
x+K_i \subset \bigcup_{l \in I(x,u)(\tau)}M_l(x) ,
\end{align*}
and $0,j \notin I(x+u,u)(\tau)$, implying for such $\tau$ (note that $M_j(x+u) = (x+u+K_j) \cap M_{Q_{\mathbb{E}}<\varphi}(x+u))$
\begin{align*}
\bigcup_{l \in I(x+u,u)(\tau)}M_l(x+u)
&\subset 
\mathbb{R}^s \big\backslash \big( M_j(x+u) \cup M_{Q_{\mathbb{E}} \geq \varphi}(x+u) \big) \\
&= \mathbb{R}^s \big\backslash \big( \big( x+u+K_j \big) \cup M_{Q_{\mathbb{E}} \geq \varphi}(x+u) \big).
\end{align*}
Combing these two inclusions we can write
\begin{align}
&\bigcup_{l \in I(x,u)(\tau)}
M_l(x) \big\backslash \bigcup_{l \in I(x+u,u)(\tau)}M_l(x+u) \nonumber \\
&\supset \big( ( x+K_i ) \cap ( x+u+K_j )\big) \cup \big( (x+K_i) 
\cap M_{Q_{\mathbb{E}}\geq \varphi}(x+u) \big). \label{SD3}
\end{align}
Note that the last expression does not depend on $\tau$ anymore. Using this and the strong convexity of $Q_{\mathbb{E}}$ on $V$ (cf. theorem \ref{theorem1})
we can continue estimating \eqref{SD1} as
\begin{align*}
& \geq \int_{  Q_{\mathbb{E}}'(x)u }^{Q'(x+u)u} 
\mu\big( ( x + K_i ) \cap (x+u+K_j) \big) \, \text{d}\tau  \\
& \geq  \mu\big( ( x + K_i ) \cap (x+u+K_j) \big)  \ \  C \ \|u\|^2.
\end{align*}
Set $C_i(u) \coloneqq \mu\big( ( x + K_i ) \cap (x+u+K_j) \big)$ and $C \coloneqq \min_i C_i$ does the deal for the first claim of the theorem due to the continuity properties of $\mu$. \\

Assuming A6 and returning to \eqref{SD1} for estimating the integrand and exploiting strong convexity of $Q_{\mathbb{E}}$ again yields
\begin{align*}
&[Q'_{D_+}(x+u)-Q'_{D_+}(x)]u \geq \\ 
& \max\big\{ \mu\big( ( x + K_i ) \cap M_{Q_{\mathbb{E}} \geq \varphi}(x+u) \big), \
\mu\big( ( x + K_i ) \cap (x+u+K_j) \big) \big\} \, C \, \|u\|^2. 
\end{align*}
For showing strong convexity of $Q_{D_+}$ it is thus sufficient to show that there is a constant $c > 0$ such that
\begin{align*}
\max\big\{ \mu\big( ( x + K_i ) \cap M_{Q_{\mathbb{E}} \geq \varphi}(x+u) \big), 
\mu\big( ( x + K_i ) \cap (x+u+K_j) \big) \big\} \geq c.
\end{align*}
Assumption A6 yields $Q_{\mathbb{E}}(y) \geq R$ for all $y \in V$ with some constant $R > 0$. Since $\varphi(0) = 0$ and $\varphi$ is continuous, there are $c_1, r_1 > 0$ such that for all $\|u\| < r_1$ in $K_i$ it is $\mu\big( ( x + K_i ) \cap M_{Q_{\mathbb{E}} \geq \varphi}(x+u) \big) \geq c_1$. Since there are only finitely many $K_i$ that $u$ can be contained in, we see that $c_1$ can actually be chosen independently of $i$.
Furthermore similar arguments as employed in the proof of lemma \ref{lemma2} show that 
$( x + K_i ) \cap (x+u+K_j) \geq c_2 > 0$ for all $u$ such that $\|u\| \geq r$. With this the proof is completed.
\end{proof}

\section{Conclusion} \label{SecConclusion}

As pointed out in section \ref{SubsecStrongConvexity}, strong convexity implies a quadratic growth condition that is instrumental in stability analysis. To illustrate this, we shall assume that the first-stage cost function is convex quadratic and view
\begin{equation}
\label{ParametricQEE}
\min_{x} \left\{ x^\top H x + h^\top x + \int_{\mathbb{R}^s} \max \lbrace g(x), \varphi(z-x) \rbrace ~\mu(\text{d}z) \; | \; x \in X \right\}
\end{equation}
as a parametric problem depending on $\mu$. Note that \eqref{rando} arises as a special case of \eqref{ParametricQEE}. Consider the parameter space $(\mathcal{M}^{1}_s, W_1)$, where
\begin{equation*}
\mathcal{M}^{1}_s := \left\{ \mu \in \mathcal{P}(\mathbb{R}^s) \; | \; \int_{\mathbb{R}^s} \max \|z\| ~\mu(\text{d}z)\right\}
\end{equation*}
is the space of Borel probability measures on $\mathbb{R}^s$ with finite moments of first order and
\begin{align*}
W_1(\mu, \nu) := \inf_{\kappa} \Big\{ \int_{\mathbb{R}^s \times \mathbb{R}^s} \|v-\tilde{v}\|~\kappa(\text{d}(v,\tilde{v})) \; | \; &\kappa \in \mathcal{P}(\mathbb{R}^s \times \mathbb{R}^s), \\ &\kappa \circ \pi_1^{-1} = \mu, \; \pi_2^{-1} = \nu \Big\}
\end{align*}
denotes the $L_1$-Wasserstein distance. In this setting, the risk neutral case arises if $g \equiv -\infty$ and has been analyzed in \cite{RoemischSchultz1991}. Following the same lines, we shall consider the optimal solution set mapping $\Psi: \mathcal{M}^1_s \rightrightarrows \mathbb{R}^n$ given by
\begin{equation*}
\Psi(\mu) := \mathrm{Argmin}_x \left\{ x^\top H x + h^\top x + \int_{\mathbb{R}^s} \max \lbrace g(x), \varphi(z-Tx) \rbrace ~\mu(\text{d}z) \; | \; x \in X \right\}.
\end{equation*}

\begin{prop}[Quantitative stability for $Q_{EE}$ and $Q_{D_+}$] \label{thStability} Assume A1-A5 as well as the following
\begin{itemize}
\item[(a)] $H$ is symmetric and positive semidefinite.
\item[(b)] $g$ is given by a sufficiently small constant $\eta$ (cf. theorem \ref{theorem2}) or by the expectation $Q_{\mathbb{E}}$. In the latter case, additionally assume A6.
\item[(c)] $\mu \in \mathcal{M}^1_s$ is such that $\Psi(\mu) \neq \emptyset$ is bounded and $T\Psi(\mu) \subseteq V$ (cf. A4).
\end{itemize}
Then there exist constants $L, \delta > 0$ such that $\Psi(\nu) \neq \emptyset$ and
\begin{equation}
\label{WassersteinHausdorff}
d_H(\Psi(\mu), \Psi(\nu)) \leq L W_1(\mu, \nu)^{\frac{1}{2}}
\end{equation}
for any $\nu \in \mathcal{M}^1_s$ satisfying $W_1(\mu, \nu) \leq \delta$.
\end{prop}

\begin{proof}
The proof for the risk neutral case given in \cite[theorem 2.7]{RoemischSchultz1991} can be extended with minor modifications whenever $g$ is Lipschitz continuous and the mapping
\begin{equation*}
x \mapsto \int_{\mathbb{R}^s} \max \lbrace g(x), \varphi(z-x) \rbrace ~\mu(\text{d}z)
\end{equation*}
is strongly convex on $V$. Thus, the corollary follows from theorem \ref{theorem2} and theorem \ref{ThSemidev}.
\end{proof}

\begin{remark}

Following the lines of \cite[theorem 2.4]{RoemischSchultz1996} the $L_1$-Wasserstein distance $W_1(\mu,\nu)^\frac{1}{2}$ in \eqref{WassersteinHausdorff} in theorem \ref{thStability} can be replaced with the subgradient distance
\begin{equation*}
d_{g}(\mu, \nu, \mathrm{cl} \; V) \coloneqq \sup \lbrace \| z^\ast \| \; | \; z^\ast \in \partial(Q_{g}^\nu - Q_{g}^\mu)(Tx), \; x \in \mathrm{cl} \; V \rbrace,
\end{equation*}
where $\partial(Q_{EE}^\nu - Q_{EE}^\mu)$ denotes the Clarke's subdifferential of
\begin{equation*}
x \mapsto \int_{\mathbb{R}^s} \max \lbrace g(x), \varphi(z-x) \rbrace ~(\nu-\mu)(\text{d}z).
\end{equation*}
Furthermore, a generalization of theorem \ref{thStability} to an appropriate Fortet-Mourier metric (see \cite[Chapter 6]{Rachev1991}) seems possible.
\end{remark}

\subsection{Future research}
The next goal in our research is to prove strong convexity results for a broader class of deviation risk measures that
can be represented as convex integral functionals. In this paper we presented a relatively self-contained first-order 
analysis of two prototypical examples. The geometry of the recourse function played an important role throughout the proofs.
A slightly different approach promising good results seems to be a second-order analysis via smooth, strongly convex approximations (see for example \cite{Azagra2013}) of the recourse-function. With this it seems to be possible to drop
assumption A6 completely and take a step towards more involved deviation risk measures. Future research might also
involve duality theory of convex integral functionals and generalized (first and second order) derivatives 
(see \cite{Rockafellar1968} for an early reference) as technical tools and a greater emphasis on stability theory.

\begin{acknowledgements}
The authors would like to thank the Deutsche Forschungsgemeinschaft for supporting the first and second author via the Collaborative Research Center TRR 154.
\end{acknowledgements}

\end{document}